\documentclass[a4paper,12pt]{article}
\usepackage{latexsym,amsmath,amsthm,amssymb}
\usepackage{a4wide}
\usepackage{hyperref}
\usepackage{marginnote}
\usepackage{color}
\usepackage{cite}
\hypersetup{
pdftitle={GN trace}   
pdfauthor={Van Hoang Nguyen},
colorlinks = true,
linkcolor = magenta,
citecolor = blue,
}

\theoremstyle{plain}
\newtheorem{theorem}{Theorem}[section]

\newtheorem{lemma}[theorem]{Lemma}

\theoremstyle{definition}

\theoremstyle{remark}

\renewcommand{\thefootnote}{\arabic{footnote}}

\def\R{\mathbb R}

\def\al{\alpha}
\def\om{\omega}
\def\Om{\Omega}
\def\de{\delta}
\def\De{\Delta} 

\def\lam{\lambda}
\def\vphi{\varphi}
\def\ep{\epsilon}
\def\na{\nabla}
\def\pa{\partial}
\def\la{\langle} 
\def\ra{\rangle} 
\def\lt{\left}
\def\rt{\right}
\def\o{\overline}

\numberwithin{equation}{section}

\title{Sharp Gagliardo--Nirenberg trace inequalities via mass transportion method and their affine versions}
\author{Van Hoang Nguyen\footnote{Institute of Research and Development, Duy Tan University, Da Nang, Vietnam.}
}

\begin{document}
\maketitle


\renewcommand{\thefootnote}{}

\footnote{Email: \href{mailto: Van Hoang Nguyen <vanhoang0610@yahoo.com>}{vanhoang0610@yahoo.com}.}

\footnote{2010 \emph{Mathematics Subject Classification\text}: 26D10, 26D15, 46E35, 42A40.}

\footnote{\emph{Key words and phrases\text}: Gagliardo--Nirenberg trace inequalities, affine Gagliardo--Nirenberg trace inequalities, Brenier map, optimal function.}

\renewcommand{\thefootnote}{\arabic{footnote}}
\setcounter{footnote}{0}

\begin{abstract}
Exploiting the mass transportation method, we prove a dual principle which implies directly the sharp Gagliardo--Nirenberg trace inequalities which was recently proved by Bolley et al. \cite{BCFGG17}. Moreover, we determine all optimal functions for these obtained sharp Gagliardo--Nirenberg trace inequalities. This settles a question left open in \cite{BCFGG17}. Finally, we use the sharp Gagliardo--Nirenberg trace inequality to establish their affine versions (i.e., the sharp affine Gagliardo--Nirenberg trace inequalities) which generalize a recent result of De N\'apoli et al. \cite{DeNapoli}. It was shown that the affine versions are stronger and imply the sharp Gagliardo--Nirenberg trace inequalities. We also determine all extremal functions for the sharp affine Gagliardo--Nirenberg trace inequalities.
\end{abstract}

\section{Introduction}
Let $\|\cdot\|$ be any norm on $\R^n$ with $n\geq 2$ and $\|\cdot\|_*$ be its dual norm, i.e., $\|x\|_* = \sup_{\|y\|\leq 1} x \cdot y$. Throughout this paper, we shall denote by
\[
\R^n_+ = \{ z= (t,x)\in \R\times \R^{n-1} : t \geq 0\},
\]
for the upper half space of $\R^n$ and by $e$ the vector $(1,0,\cdots,0)\in \R^n$. Then $\pa \R^n_+ = \{(0,x) : x\in \R^{n-1}\}$. For $p\in (1,n)$ and $q>1$, we denote $\mathcal D^{p,q}(\R^n_+)$ the set of all functions $f \in L^q(\R^n_+)$ such that its distributional gradient $\na f$ belongs to $L^p(\R^n_+)$. In recent paper \cite{BCFGG17}, Bolley et al., have proved the following Gagliardo--Nirenberg trace inequalities on $\R^n_+$ for $a\geq n > p>1$
\begin{equation}\label{eq:GNtrace}
\lt(\int_{\pa \R^n_+} |f|^{p\frac{a-1}{a-p}} dx\rt)^{\frac{a-p}{p(a-1)}} \leq D_{n,p,a} \lt(\int_{\R^n_+} \|\na f\|_*^p dz\rt)^{\frac\theta p} \lt(\int_{\R^n_+} |f|^{p\frac{a-1}{a-p}} dz\rt)^{(1-\theta)\frac{a-p}{p(a-1)}}
\end{equation}
for any smooth function $f\in \mathcal D^{p,\frac{p(a-1)}{a-p}}(\R^n_+)$, where $\theta$ is determined by the scaling invariant of \eqref{eq:GNtrace} and given by
\begin{equation}\label{eq:theta}
\theta = \frac{a-p}{p(a-n) + n-p}.
\end{equation}
The inequality \eqref{eq:GNtrace} is sharp, and the optimal constant $D_{n,p,a}$ is reached by function 
\[
f(z) = h_p(z) = \|z + e\|^{-\frac{a-p}{p-1}},\qquad z\in \R^n_+.
\]
It is an analogue of the sharp Gagliardo--Nirenberg inequalities due to Del Pino and Dolbeault \cite{DD2002,DD2003} (see \cite{CNV2004} for another proof by using mass transportation method) in the trace case. It contains the sharp Sobolev trace inequality as a special case with $a =n$ which was proved by Nazaret \cite{Nazaret2006}. In \cite{BCFGG17}, the inequality \eqref{eq:GNtrace} was derived from a new kind of Borell--Brascamp--Lieb inequality (see Theorem $6$ in \cite{BCFGG17}). The latter inequality also leads to the sharp Sobolev inequalities \cite{Aubin1976,Tal1976}, the sharp Gagliardo--Nirenberg inequalities \cite{CNV2004,DD2002,DD2003} and the sharp Sobolev trace inequalities \cite{Nazaret2006}. This method gives a new bridge between the geometric point of view of the Brunn--Minkowski inequality and the functional point of view of the Sobolev type inequalities. Notice that the relation between Brunn--Minkowski inequality and the Sobolev type inequalities was observed by Bobkov and Ledoux. It has shown by Bobkov and Ledoux \cite{BL00,BL08} that the Sobolev type inequalities (such as logarithmic Sobolev inequality and Sobolev inequalities) can be reached by a functional version of Brunn--Minkowski inequality known as Borell--Brascamp--Lieb inequality due to Borell \cite{Borell75} and Brascamp and Lieb \cite{BL76}. However, we can not use the standard functional version of the Borell--Brascamp--Lieb inequality (see \cite{BL08}). It was asked by Bobkov and Ledoux whether the Sobolev inequalities can be proved directly from a new kind of Borell--Brascamp--Lieb inequality. The paper \cite{BCFGG17} answers affirmatively this question. However, it left open the question on the extremal functions for the inequality \eqref{eq:GNtrace}. It was conjectured in \cite{BCFGG17} that the equality holds in \eqref{eq:GNtrace} if and only if $f =h_p$ up to translations, dilations and multiplicative constants. 

The aim of this paper is to give a new and direct proof of the inequality\eqref{eq:GNtrace} based on the mass transportation method. It is well known that the mass transportation method is an useful tool to prove some inequalities (in sharp form) both in analysis and geometry (such as Brunn--Minkowski inequality and its functional versional, the isoperimetric inequality, the sharp Sobolev and sharp Gagliardo--Nirenberg inequalities \cite{CNV2004,Nguyen2015}, the sharp Sobolev trace inequality \cite{Nazaret2006}, and some Gaussian type inequalities \cite{Cordero2002}). In fact, we shall use the mass transportation method to prove a dual principle (see Theorem \ref{Maintheorem1} below) which immediately implies the inequality \eqref{eq:GNtrace}. Furthermore, tracing back to estimates in this method, we cans show that the optimal constant $D_{n,p,a}$ is reached only by $h_p$ up to translations, dilations and multiplicative constants (see Theorem \ref{Maintheorem2} below). This settles a question left open in \cite{BCFGG17}. Note that the case $a =n$ (corresponding to the Sobolev trace inequality) was recently treated by Maggi and Neumayer \cite{MN2017} by using again the mass transportation method. For a function $f$ on $\R^n_+$, we denote $\|f\|_p$ the $L^p-$norm of $f$ with respect to Lebesgue measure on $\R^n_+$, i.e.,
\[
\|f\|_p = \lt(\int_{\R^n_+} |f(z)|^p dz\rt)^{\frac 1p}.
\]
Our main results read as follows.

\begin{theorem}\label{Maintheorem1}
Let $a\geq n > p >1$, and let $f$ and $g$ be functions in $\mathcal D^{p,\frac{p(a-1)}{a-p}}(\R^n_+)$ such that $\|f\|_p = \|g\|_p = \|h_p\|_p$. Then we have
\begin{align}\label{eq:dual}
a\int_{\R^n_+}& |g|^{\frac{p(a-1)}{a-p}} dz - (a-1)\int_{\R^n_+} \|z+e\|^{\frac p{p-1}} |g|^{\frac{ap}{a-p}} dz \notag\\
&\leq (a-n) \int_{\R^n_+} |f|^{\frac{p(a-1)}{a-p}} dz +\frac{(a-1)(p-1)^{p-1}}{(a-p)^p}\int_{\R^n_+} \|\na f\|_*^p dz -\int_{\pa \R^n_+} |f|^{\frac{p(a-1)}{a-p}} dx,
\end{align}
with equality if $f =g =h_p$. As a consequence, we have the following dual principle
\begin{align}\label{eq:dualprin}
\sup_{\|g\|_{\frac{ap}{a-p}} = \|h_p\|_{\frac{ap}{a-p}}}&\Big(a\int_{\R^n_+}|g|^{\frac{p(a-1)}{a-p}} dz - (a-1)\int_{\R^n_+} \|z+e\|^{\frac p{p-1}} |g|^{\frac{ap}{a-p}} dz\Big) \notag\\
&= \inf_{\|f\|_{\frac{ap}{a-p}} = \|h_p\|_{\frac{ap}{a-p}}}\Bigg((a-n) \int_{\R^n_+} |f|^{\frac{p(a-1)}{a-p}} dz +\frac{(a-1)(p-1)^{p-1}}{(a-p)^p}\int_{\R^n_+} \|\na f\|_*^p dz\notag\\
&\hspace{7cm} -\int_{\pa \R^n_+} |f|^{\frac{p(a-1)}{a-p}} dx\Bigg),
\end{align}
with $h_p$ is extremal in both variational problems. Furthermore, the Gagliardo--Nirenberg trace inequality \eqref{eq:GNtrace} holds.
\end{theorem}

Concerning to the optimal functions for the inequality \eqref{eq:GNtrace}, we have

\begin{theorem}\label{Maintheorem2}
A function $f\in \mathcal D^{p,\frac{p(a-1)}{a-p}}(\R^n_+)$ is optimal in the Gagliardo--Nirenberg trace inequality \eqref{eq:GNtrace} if and only if 
\[
f(t,x) = c h_p(\lam(t,x-x_0)), \qquad (t,x) \in \R^n_+,
\]
for some $c\in \R$, $\lambda >0$ and $x_0\in \R^{n-1}$.
\end{theorem}

As mentioned above, the proofs of Theorems \ref{Maintheorem1} and \ref{Maintheorem2} are based on the mass transportation method. We shall briefly recall some facts from this method. Let $\mu$ and $\nu$ be two Borel measures on $\R^n$ of the same total measure such that $\mu$ is absolutely continuous with respect to Lebesgue measure. Then there exists a convex function $\varphi$ on $\R^n$ such that
\begin{equation}\label{eq:push-forward}
 \int_{\R^n} b(y)d\nu(y) = \int_{\R^n} b(\nabla\varphi(z))d\mu(z) .
\end{equation}
for every bounded or positive, Borel measurable function $b:\R^n\to \R$ (see \cite{Brenier1991, McCann1995}). Furthermore, $\overline{\nabla\varphi(\text{\rm supp} \,\mu)} = \text{\rm supp}\, \nu$ and $\nabla \varphi$ is uniquely determined $d\mu-$almost everywhere. We call $\nabla\varphi$ the \emph{Brenier map} which transports $\mu$ to $\nu$. See \cite{Vil2009} for a review and dicussion of existing proofs of this map. Since $\varphi$ is convex, it is differentiable almost everywhere on its domain $\{ \varphi < \infty\}$; in particular, it is differentiable $d\mu-$almost everywhere. In additionally, if $\mu$ and $\nu$ are absolutely continuous with respect to the Lebesgue measure on $\R^n$, with densities $F$ and $G$ respectively. Then $\na\vphi$ satisfies
\begin{equation}\label{eq:push-forward1}
\int_{\R^n} b(y)G(y)dy = \int_{\R^n} b(\nabla\varphi(z))F(z)dz
\end{equation}
for every bounded or positive, Borel measurable function $b:\R^n\to \R$. It is well-known that $\na \varphi$ solves the following Monge--Amp\`ere equation in the $F(x)dx-$almost everywhere sense (see \cite{McCann1997})
\begin{equation}\label{eq:MAequation}
F(z) = G(\nabla\varphi(z))\det\,(D^2_A\varphi(z))
\end{equation} 
where $D^2_A\varphi$ is Hessian of $\vphi$ in Aleksandrov sense, i.e, as the absolutely continuous part of the distributional Hessian of the convex function $\varphi$.

In the proofs below, it is required to use H\"older's and Young's inequality. We recall them below. Let $\|\cdot\|$ be a norm on $\R^n$, $B$ its unit ball and $\|\cdot\|_*$ its dual norm. Then for any $\lambda >0$, Young's inequality holds
\begin{equation}\label{eq:Youngineq}
X\cdot Y \leq \frac{\lambda^{-p}}p \|X\|_*^p + \frac{\lambda^q}q \|Y\|^q, \quad\, q = \frac{p}{p-1}.
\end{equation}
For $X:\R^n\to (\R^n,\|\cdot\|_*)$ in $L^p$ and $Y: \R^n\to (\R^n,\|\cdot\|)$ in $L^q$, integration of~\eqref{eq:Youngineq} and optimization in $\lambda$ gives H\"older's inequality in the form
\begin{equation}\label{eq:Holderineq}
\int_{\R^n} X(z) \cdot Y(z)\, dz \leq \left(\int_{\R^n}\|X(z)\|_*^p \, dz\right)^{\frac1p}\, \left(\int_{\R^n} \|Y(z)\|^q\, dz\right)^{\frac1q}.
\end{equation}

Since $\|\cdot\|$ is a Lipschitz function on $\R^n$, it is differentiable almost everywhere. When $0\not=z\in \R^n$ is such a point of differentiability, the gradient of the norm at $z$ is the unique vector $z^* = \nabla(\|\cdot\|)(z)$ such that
\begin{equation}\label{eq:gradientofnorm}
\|z^*\|_* = 1, \quad\quad z\cdot z^* = \|z\| = \sup_{\|y\|_* = 1}z\cdot y.
\end{equation}
These equalities will be used to verify the extremality of $h_p$.

As an application of the sharp Gagliardo--Nirenberg trace inequalities \eqref{eq:GNtrace}, we shall establish the sharp affine Gagliardo--Nirenberg trace inequalities in the half space $\R^n_+$. We mention here that the sharp affine Sobolev trace inequality was recently established by De N\'apoli et al., \cite{DeNapoli}. Let $p \in (1,n)$, let us denote by $\mathcal D^p(\R^n_+)$ the set of all functions $f$ in $\R^n_+$ such that its distribution gradient belongs to $L^p(\R^n_+)$ and the set $\{|f|> s\}$ has finite measure for any $s >0$. It is known that $\mathcal D^p(\R^n_+) \hookrightarrow L^{\frac{np}{n-p}}(\R^n_+)$ (see \cite{Nguyen2015}). Let $f \in \mathcal D^p(\R^n_+)$, we set
\[
\mathcal E_p(f) = c_{n-1,p} \lt(\int_{S^{n-2}}\lt(\int_{\R^n_+} |\la \na f(x), \xi\ra|^p dx\rt)^{-\frac{n-1}p} d\xi\rt)^{-\frac1{n-1}},
\]
where $S^{n-2}$ denotes the unit sphere in $\{0\}\times \R^{n-1}$ and
\[
c_{n,p} = (n \om_n)^{\frac1n}\lt(\frac{n \om_n \om_{p-1}}{2 \om_{n+p-2}}\rt)^{\frac1p},\qquad \om_s = \frac{\pi^{\frac s2}}{\Gamma(1+ \frac s2)}, \quad s >0.
\]
It was proved in \cite{DeNapoli} that if $p \in (1,n), n\geq 3$ and $f \in \mathcal D^p(\R^n_+)$, we have
\begin{equation}\label{eq:affineSobolevtrace}
\lt(\int_{\pa \R^n_+} |f(0,x)|^{\frac{p(n-1)}{n-p}} dx\rt)^{\frac{n-p}{n-1}} \leq p \mathcal B_{n,p} \mathcal E_p(f)^{p-1} \lt(\int_{\R^n_+} |\pa_t f|^p dxdt\rt)^{\frac1p},
\end{equation}
where 
\[
\mathcal B_{n,p} = \pi^{-\frac{p-1}2}\lt(\frac{(p-1)^{\frac{p-1}p}}{n-p}\rt)^{p-1} \lt(\frac{\Gamma(n) \Gamma(\frac{n+1}2)}{(n-1) \Gamma(\frac{n-1}p) \Gamma(\frac{n(p-1)+1}{p})}\rt)^{\frac{p-1}{n-1}}.
\]
The inequality \eqref{eq:affineSobolevtrace} is sharp and there is equality in \eqref{eq:affineSobolevtrace} if 
\begin{equation}\label{eq:extremal}
f(t,x) = \pm \lt((\lam t + \de)^p + |A(x-x_0)|^p\rt)^{-\frac{n-p}{p(p-1)}},
\end{equation}
for some constant $\lam, \de >0$, $x_0 \in R^{n-1}$ and $A \in {\rm GL}_{n-1}$.

The proof of \eqref{eq:affineSobolevtrace} given in \cite{DeNapoli} is based on the sharp Sobolev trace inequality due to Nazaret \cite{Nazaret2006} and the sharp affine $L_p$ isoperimetric inequalities due to Lutwak, Yang and Zhang \cite{LYZ2000} (more precisely, the $L_p$ Busemann--Petty centroid inequality). This approach was already used in \cite{Haddad} to give a new proof a the sharp affine $L_p$ Sobolev inequalities due to Lutwak, Yang and Zhang \cite{LYZ2002,LYZ2006}. The advantage of this new method is that it is more simple than the ones of Lutwak, Yang and Zhang, and avoids the use of the solution of the $L_p$ Minkowski problem \cite{LYZ2004} which is still open in general. It was further developed by the author in \cite{Nguyen2016} to prove a general affine P\'olya--Szeg\"o principle and its equality characterization (i.e., a Brothers--Ziemer type result). Applying these results, the author gave a new proof of general affine Sobolev inequalities and classified the extremal functions for these inequalities which remains open from the previous works (e.g., see \cite{HS09Sob,HSX,Wang}). It is remakable by Jensen's inequality and Young's inequality that 
\[
\mathcal E_p(f)^p \leq \int_{\R^n_+} |\na_x f(t,x)|^p dx dt
\]
where $\na_x f$ is the gradient of $f$ in variables $x\in \R^{n-1}$. Hence, we have from \eqref{eq:affineSobolevtrace}
\begin{equation}\label{eq:EuSob}
\lt(\int_{\pa \R^n_+} |f(0,x)|^{\frac{p(n-1)}{n-p}} dx\rt)^{\frac{n-p}{n-1}} \leq (p-1)^{\frac{p-1}p} \mathcal B_{n,p} \int_{\R^n_+} \lt(|\na_x f|^p + |\pa_t f|^p\rt) dtdx,
\end{equation}
which is exactly the Sobolev trace inequality of Nazaret with the norm defined by $\|(t,x)\|= (|t|^{\frac p{p-1}} + |x|^{\frac p{p-1}})^{\frac{p-1}p}$. In the special case $p=2$, the inequality \eqref{eq:EuSob} reduces to the famous sharp Sobolev trace inequality due to Beckner \cite{Beckner} and Escobar \cite{Escobar}. In particular, the inequality \eqref{eq:affineSobolevtrace} is stronger than and implies the sharp Sobolev trace inequality of Beckner and Escobar. 

Concerning to the extremal functions for the inequality \eqref{eq:affineSobolevtrace}, it was proved in \cite{DeNapoli} (see Theorem $1.1$) that the extremal functions for \eqref{eq:affineSobolevtrace} in the case $p=2$ are only functions of the form \eqref{eq:extremal} with $p=2$. This follows from the characterization of extremal functions for the Sobolev trace inequality due to Beckner and Escobar, and the equality case in the $L_p$ Busemann--Petty centroid inequality (see \cite{LYZ2000}). The case $p\not=2$ was left open since the lack of characterization of extremal functions for the Sobolev trace inequality due to Nazaret \cite{Nazaret2006}. However, as mentioned above, the extremal functions for the Sobolev trace inequality already was treated by Maggi and Neumayer (see \cite{MN2017}). Hence, as for the case $p =2$, we can conclude that all extremal functions for the inequality \eqref{eq:affineSobolevtrace} are of the form \eqref{eq:extremal}. Another important consequence of Theorem \ref{Maintheorem2} is that the inequality \eqref{eq:affineSobolevtrace} does not imply the classical sharp Sobolev trace inequality
\[
\lt(\int_{\pa \R^n_+} |f(0,x)|^{\frac{p(n-1)}{n-p}} dx\rt)^{\frac{n-p}{p-1}} \leq \mathcal K_{n,p} \int_{\R^n_+} |\na f(t,x)|^p dx dt,
\]
with $\mathcal K_{n,p}$ is the sharp constant. We refer the reader to the end of the paper \cite{DeNapoli} for more detail discussions about this fact.

Notice that the inequality \eqref{eq:affineSobolevtrace} possesses an important affine invariant property as described below. Denote by ${\rm GL}_{n,+}$ the set of matrices having the form
\[
\begin{bmatrix}
\lam &0 &\cdots & 0\\
0\\
\vdots & & B\\
0
\end{bmatrix}
\]
where $\lam >0$ and $B \in {\rm GL}_{n-1}$. For any $A \in {\rm GL}_{n,+}$ and $f\in \mathcal D^p(\R^n_+)$, we define $g(t,x) = f(A(t,x)) = f(\lam t, Bx)$, the simple computations show that
\[
\lt(\int_{\pa \R^n_+} |g(0,x)|^{\frac{p(n-1)}{n-p}} dx \rt)^{\frac{n-p}{n-1}}= |{\rm det} B|^{-\frac{n-p}{n-1}} \lt(\int_{\pa \R^n_+} |f(0,x)|^{\frac{p(n-1)}{n-p}} dx\rt)^{\frac{n-p}{n-1}},
\]
and
\[
\mathcal E_p(g)^{p-1} \lt(\int_{\R^n_+} |\pa_t g(t,x)|^p dx dt\rt)^{\frac1p} =|{\rm det} B|^{-\frac{n-p}{n-1}} \mathcal E_p(f)^{p-1} \lt(\int_{\R^n_+} |\pa_t f(t,x)|^p dx dt\rt)^{\frac1p}.
\]
Hence, the inequality \eqref{eq:affineSobolevtrace} is ${\rm GL}_{n,+}$ invariant. It is easy to see that ${\rm GL}_{n,+}$ is a subgroup of ${\rm SGL}_{n,+}$ consisting of matrices of the form
\[
\begin{bmatrix}
\lam &0 &\cdots & 0\\
a_1\\
\vdots & & B\\
a_{n-1}
\end{bmatrix}
\]
where $\lam >0$, $a_1,\ldots,a_{n-1} \in \R$, and $B \in {\rm GL}_{n-1}$. Obviously, $\R^n_+$ is invariant under the action of elements in ${\rm SGL}_{n,+}$. This suggests that the inequality \eqref{eq:affineSobolevtrace} can be strengthened to a ${\rm SGL}_{n,+}$ invariant form. We will show that this is the case. Indeed, using Theorems \ref{Maintheorem1} and \ref{Maintheorem2} and the sharp $L_p$ Busemann--Petty centroid inequality, we shall prove a family of the affine sharp Gagliardo--Nirenberg trace inequalities in $\R^n_+$ which contains the inequality \eqref{eq:affineSobolevtrace} as a special case. To state our next result, let us introduce a new quantity on $\mathcal D^p(\R^n_+)$ which is ${\rm SGL}_{n,+}$ invariant. For $f \in \mathcal D^p(\R^n_+)$, we denote
\[
\mathcal A_p(f)^p =c_{n,p,a}^p (\mathcal E_p(f))^{p-1 -\frac{(p-1)(a-n)}{a-1}} \inf_{\vec{a} \in \R^{n-1}} \lt(\int_{\R^n_+} |D_{(1,\vec{a})} f|^p dx dt\rt)^{\frac1p(1 +\frac{(p-1)(a-n)}{a-1})}.
\]
where
\[
c_{n,p,a} =\lt[\frac{p(a-1)}{(p-1)(n-1)} \lt(\frac{(p-1)(n-1)}{p(a-n) + n-1}\rt)^{\frac1p + \frac{(p-1)(a-n)}{p(a-1)}}\rt]^{\frac1p},
\] 
and $D_{(1,\vec{a})}f$ denotes the derivative of $f$ in the direction $(1,\vec{a})$, i.e.,
\[
D_{(1,\vec{a})}f(t,x) = \pa_t f(t,x) + \la \vec{a}, \na_x f(t,x)\ra.
\]
Given a matrix $A \in {\rm SGL}_{n,+}$ which is determined by a $\lam >0, \vec{a} =(a_1,\ldots,a_{n-1}) \in \R^{n-1}$ and $B \in {\rm GL}_{n-1}$, define $g(t,x) = f(A(t,x)) = f(\lam t, \vec{a} t + B x)$. It is not hard to see that
\begin{equation}\label{eq:affinechange}
\mathcal A_p(g) = |{\rm det} B|^{- \frac{a-p}{p(a-1)}} \lam^{\frac{(p-1)(a-n)}{p(a-1)}}\mathcal A_p(f).
\end{equation}
The affine version of the Gagliardo--Nirenberg trace inequality \eqref{eq:GNtrace} is stated as follows.
\begin{theorem}\label{affineGNtrace}
Let $a \geq n > p >1$. Then it holds
\begin{equation}\label{eq:affineGNtrace}
\lt(\int_{\pa \R^n_+} |f(0,x)|^{\frac{p(a-1)}{a-p}} dx\rt)^{\frac{a-p}{p(a-1)}} \leq \mathcal D_{n,a,p}\,  \mathcal A_p(f)^{\theta} \lt(\int_{\R^n_+} |f(t,x)|^{p\frac{a-1}{a-p}} dxdt\rt)^{(1-\theta)\frac{a-p}{p(a-1)}}
\end{equation}
for any $f \in \mathcal D^{p, \frac{p(a-1)}{a-p}}(\R^n_+)$, where $\theta$ is given by \eqref{eq:theta} and $\mathcal D_{n,a,p}$ is the sharp constant whose value is given by
\begin{align}\label{eq:sharpconstantaffineGN}
\mathcal D_{n,a,p}& = \lt(\frac{p-1}{a-p}\rt)^{\theta} \lt(\frac{p(a-n) + n-p}{p-1}\rt)^{\frac\theta p + (1-\theta) \frac{a-p}{p(a-1)}} \times \notag\\
&\qquad\qquad\qquad \qquad\qquad \times \lt(\frac{\Gamma(a) \Gamma(\frac{n+1}2)}{\pi^{\frac{n-1}2}(a-1) \Gamma(\frac{n-1}q +1) \Gamma(a- \frac{n-1}q -1)}\rt)^{\frac{\theta}{q(a-1)}}.
\end{align}
Moreover, the equality in \eqref{eq:affineGNtrace} holds if and only if 
\begin{equation}\label{eq:extremalGN}
f(t,x) = c\, h_p(A(t, x-x_0))
\end{equation}
for some $c \in \R$, $x_0\in \R^{n-1}$ and $A \in {\rm SGL}_{n,+}$, where $h_p(t,x) =((1+t)^q + |x|^q)^{-\frac{a-p}{p-1}}$.
\end{theorem}
In the case $a =n$, we then have $\theta =1$ and
\begin{align*}
\mathcal A_p(f)^p& = \frac{p}{(p-1)^{\frac{p-1}p}} \mathcal E_p(f)^{p-1} \inf_{\vec{a} \in \R^{n-1}} \|D_{(1,\vec{a})} f\|_p \leq \frac{p}{(p-1)^{\frac{p-1}p}} \mathcal E_p(f)^{p-1} \|\pa_t f\|_p.
\end{align*}
Hence, the inequality \eqref{eq:affineGNtrace} is stronger than the inequality \eqref{eq:affineSobolevtrace} when $a =n$. Moreover, by Young inequality and $\mathcal E_p(f)^p \leq \int_{\R^n_{+}} |\na_x f|^p dx dt$, we can easily show that 
\[
\mathcal A_p(f)^p \leq \int_{\R^n_+} (|\pa _t f|^p + |\na_x f|^p) dx dt,
\]
Hence the inequality \eqref{eq:affineGNtrace} is stronger than the inequality \eqref{eq:GNtrace} with the norm $\|(t,x)\| = \lt(|t|^q + |x|^q\rt)^{\frac1q}$. In particular, for $p=q=2$, the inequality \eqref{eq:affineGNtrace} is stronger than the inequality \eqref{eq:GNtrace} with the euclidean norm. Furthermore, for the case $p\not=2$, the inequality \eqref{eq:affineGNtrace} does not imply the sharp Gagliardo--Nirenberg trace inequality \eqref{eq:GNtrace} with the euclidean norm because the set of extremal functions of two these inequalities are different.

The rest of this paper is organized as follows. In the next section, we use the mass transportation method to prove Theorem \ref{Maintheorem1}. Theorem \ref{Maintheorem2} will be proved in section 3 by refining the step in proof of Theorem \ref{Maintheorem1}. In section 4, we prove the sharp affine Gagliardo--Nirenberg trace inequality and determine the extremal functions (i.e., proving Theorem \ref{affineGNtrace}).

\section{Proof of Theorem \ref{Maintheorem1}}
In this section, we give the proof of Theorem \ref{Maintheorem1} via the mass transportation method. Remark that $\|\na |f|\|_* \leq \|\na f\|_*$ for $f \in \mathcal D^{p,\frac{p(a-1)}{a-p}}(\R^n_+)$, then it suffices to prove Theorem \ref{Maintheorem1} for nonnegative functions. We start by proving the inequality \eqref{eq:dual} for $f\in \mathcal D^{p,\frac{a p}{a-p}}(\R^n_+)$ and $g\in L^{\frac{ap}{a-p}}(\R^n_+)$ which are compactly supported, nonnegative, smooth functions such that $\|f\|_{\frac{ap}{a-p}} = \|g\|_{\frac{ap}{a-p}} = \|h_p\|_{\frac{ap}{a-p}}$ (Note that this assumption makes sense since $\mathcal D^{p,\frac{p(a-1)}{a-p}}(\R^n_+) \hookrightarrow L^{\frac{ap}{a-p}}(\R^n_+)$ by the Gagaliardo--Nirenberg inequality on half space). The extra assumptions imposed on $f$ and $g$ ensure some regularity properties of the Brenier map pushing $f^{\frac{ap}{a-p}} dz$ forward to $g^{\frac{ap}{a-p}} dz$. Indeed, let $\varphi$ be such a Brenier map. As was shown in \cite{MV2005}, the convex function $\vphi$ can be assumed to have as its domain the whole space $\R^n$ since $g$ is compactly supported in $\R^n_+$. Thus $\na \vphi$ has locally bounded variation in $\R^n_+$ up to the boundary $\pa \R^n_+$. By \eqref{eq:MAequation}, $\na\vphi$ satisfies the quation 
\[
f^{\frac{ap}{a-p}}(z) = g^{\frac{ap}{a-p}}(\na \vphi(z))\, \text{\rm det}\, (D^2_A\vphi(z)),
\]
for $f^{\frac{ap}{a-p}} dz$-a.e. $z$. The preceding equation together with the definition of mass transportation gives
\begin{align}\label{eq:useBrenier}
\int_{\R^n_+} g^{\frac{p(a-1)}{a-p}} dz &= \int_{\R^n_+} g^{\frac{ap}{a-p}} g^{-\frac{p}{a-p}} dz\notag\\
&=\int_{\R^n_+} f(z)^{\frac{ap}{a-p}} g(\na\vphi(z))^{-\frac{p}{a-p}} dz\notag\\
&=\int_{\R^n_+} f(z)^{\frac{p(a-1)}{a-p}} \lt[\text{\rm det}\,(D^2_A \vphi(z))\rt]^{\frac1a} dz.
\end{align}
By the arithmetic--geometric inequality, we have
\begin{align}\label{eq:AGinequality}
\lt[\text{\rm det}\,(D^2_A \vphi(z))\rt]^{\frac1a} &=\lt[\lt[\text{\rm det}\,(D^2_A \vphi(z))\rt]^{\frac1n}\rt]^{\frac na} \notag\\
&\leq \frac{n}{a} \lt[\text{\rm det}\,(D^2_A \vphi(z))\rt]^{\frac1n} + \frac{a-n}a\notag\\
&\leq \frac1a \Delta_A\vphi(z) + \frac{a-n}a.
\end{align}
Plugging \eqref{eq:AGinequality} into \eqref{eq:useBrenier} and using the inequality $\De_A \vphi \leq \De_{\mathcal D'}\vphi$ where $\De_{\mathcal D'}\vphi$ denotes the distributional Laplacian of $\vphi$, we get
\begin{align}\label{eq:AlekLap}
\int_{\R^n_+} g^{\frac{p(a-1)}{a-p}} dz &\leq \frac{a-n}a \int_{\R^n_+} f(z)^{\frac{p(a-1)}{a-p}} dz + \frac1a \int_{\R^n_+} f(z)^{\frac{p(a-1)}{a-p}}\, \Delta_A\vphi(z) dz\notag\\
&\leq\frac{a-n}a \int_{\R^n_+} f(z)^{\frac{p(a-1)}{a-p}} dz + \frac1a \int_{\R^n_+} f(z)^{\frac{p(a-1)}{a-p}}\, \Delta_{\mathcal D'}\vphi(z) dz\notag\\
&\leq \frac{a-n}a \int_{\R^n_+} f(z)^{\frac{p(a-1)}{a-p}} dz + \frac1a \int_{\R^n_+} f(z)^{\frac{p(a-1)}{a-p}}\, \Delta_{\mathcal D'}\psi(z) dz,
\end{align}
where $\psi(z) = \vphi(z) + e\cdot z$ which has the same distributional Laplacian as $\vphi$. Note that $f$ is smooth and has compact support, then we can use the integration by parts formula for functions of bounded variation. This leads to
\begin{equation}\label{eq:IBPBV0}
\int_{\R^n_+} f(z)^{\frac{p(a-1)}{a-p}}\, \Delta_{\mathcal D'}\psi(z) dz = -\frac{p(a-1)}{a-p} \int_{\R^n_+} \na f \cdot \na \psi f^{\frac{a(p-1)}{a-p}} dz - \int_{\pa \R^n_+} f^{\frac{p(a-1)}{a-p}} \na\psi \cdot e dx,
\end{equation}
since $-e$ is the exterior normal vector to $\pa \R^n_+$. By definition of the mass transportation, $\na\vphi \in \R^n_+$ which means that $\na \vphi \cdot e \geq 0$ on $\pa \R^n_+$. Since $e \cdot e =1$, we get
\begin{equation}\label{eq:IBPBV}
\int_{\R^n_+} f(z)^{\frac{p(a-1)}{a-p}}\, \Delta_{\mathcal D'}\psi(z) dz \leq -\frac{p(a-1)}{a-p} \int_{\R^n_+} \na f \cdot \na \psi f^{\frac{a(p-1)}{a-p}} dz - \int_{\pa \R^n_+} f^{\frac{p(a-1)}{a-p}} dx.
\end{equation}
Combining \eqref{eq:AlekLap} and \eqref{eq:IBPBV}, we arrive
\begin{align}\label{eq:IBPBV*}
a\int_{\R^n_+} g^{\frac{p(a-1)}{a-p}} dz &\leq (a-n) \int_{\R^n_+} f(z)^{\frac{p(a-1)}{a-p}} dz -\frac{p(a-1)}{a-p} \int_{\R^n_+} \na f \cdot \na \psi f^{\frac{a(p-1)}{a-p}} dz\notag\\
&\qquad\qquad -\int_{\pa \R^n_+} f^{\frac{p(a-1)}{a-p}} dx.
\end{align}
Applying H\"older inequality and the definition of mass transportation, we obtain
\begin{align}\label{eq:YoungBre}
- \int_{\R^n_+} \na f \cdot \na \psi f^{\frac{a(p-1)}{a-p}} dz & \leq \lt(\int_{\R^n_+} \|\na f\|_*^p dz\rt)^{\frac1p}\lt( \int_{\R^n_+} \|\na \psi\|^{\frac p{p-1}} f^{\frac {ap}{a-p}} dz\rt)^{\frac{p-1}p}\notag\\
 &=\lt(\int_{\R^n_+} \|\na f\|_*^p dz\rt)^{\frac1p}\lt( \int_{\R^n_+} \|z+e\|^{\frac p{p-1}} g^{\frac {ap}{a-p}} dz\rt)^{\frac{p-1}p}.
\end{align}
Inserting \eqref{eq:YoungBre} into \eqref{eq:IBPBV*}, we get
\begin{align}\label{eq:IBPBV**}
a\int_{\R^n_+} g^{\frac{p(a-1)}{a-p}} dz &\leq  -\int_{\pa \R^n_+} f^{\frac{p(a-1)}{a-p}} dx + (a-n) \int_{\R^n_+} f(z)^{\frac{p(a-1)}{a-p}} dz \notag\\
&\qquad + \frac{p(a-1)}{a-p}\lt(\int_{\R^n_+} \|\na f\|_*^p dz\rt)^{\frac1p}\lt( \int_{\R^n_+} \|z+e\|^{\frac p{p-1}} g^{\frac {ap}{a-p}} dz\rt)^{\frac{p-1}p}.
\end{align}
Since the Brenier map $\na\vphi$ does not appear in \eqref{eq:IBPBV**}, we can remove the compactness and smoothness assumptions imposed on $f$ and $g$ at this stage by density argument. The estimate \eqref{eq:IBPBV**} hence remains hold for any nonnegative functions $f\in \mathcal D^{p,\frac{p(a-1)}{a-p}}(\R^n_+)$ and $g\in L^{\frac{ap}{a-p}}(\R^n_+)$ satisfying $\|f\|_{\frac{ap}{a-p}} = \|g\|_{\frac{ap}{a-p}} = \|h_p\|_{\frac{ap}{a-p}}$. Let us remark that when $f =g =h_p$ then $\na\vphi(z) = z$ which implies the equality in \eqref{eq:IBPBV**}.

By Young inequality, we have
\begin{align}\label{eq:Young1}
\lt(\int_{\R^n_+} \|\na f\|_*^p dz\rt)^{\frac1p}&\lt( \int_{\R^n_+} \|z+e\|^{\frac p{p-1}} g^{\frac {ap}{a-p}} dz\rt)^{\frac{p-1}p} \notag\\
&\leq \frac{a-p}p \int_{\R^n_+} \|z+e\|^{\frac p{p-1}} g^{\frac {ap}{a-p}} dz + \frac1p \frac{(p-1)^{p-1}}{(a-p)^{p-1}} \int_{\R^n_+} \|\na f\|_*^p dz,
\end{align}
with equality if $f=g=h_p$. Therefore, the inequality \eqref{eq:dual} follows from the estimates \eqref{eq:IBPBV**} and \eqref{eq:Young1}, with equality if $f=g=h_p$. This leads to the dual principle \eqref{eq:dualprin}.

It remains to prove the Gagliardo--Nirenberg trace inequality \eqref{eq:GNtrace}. Since
\[
\sup_{\|g\|_{\frac{ap}{a-p}} = \|h_p\|_{\frac{ap}{a-p}}}\Big(a\int_{\R^n_+}|g|^{\frac{p(a-1)}{a-p}} dz - (a-1)\int_{\R^n_+} \|z+e\|^{\frac p{p-1}} |g|^{\frac{ap}{a-p}} dz\Big) = \int_{\R^n_+} h_p^{\frac{p(a-1)}{a-p}} dz =:A.
\]
Then we have
\begin{equation}\label{eq:nonhomo1}
\int_{\pa \R^n_+} |f|^{\frac{p(a-1)}{a-p}} dx \leq (a-n) \int_{\R^n_+} |f|^{\frac{p(a-1)}{a-p}} dz +\frac{(a-1)(p-1)^{p-1}}{(a-p)^p}\int_{\R^n_+} \|\na f\|_*^p dz -A,
\end{equation}
with equality if $f =h_p$. Using again Young inequality, we get
\begin{align}\label{eq:Young2}
\frac{(a-1)(p-1)^{p-1}}{(a-p)^p}\int_{\R^n_+} \|\na f\|_*^p dz&=\frac{(a-1)^{\frac{a-p}{a-1}}(p-1)^{\frac{a(p-1)}{a-1}}}{(a-p)^p}\frac{\int_{\R^n_+}\|\na f\|_*^p dz}{A^{\frac{p-1}{a-1}}} \lt(\frac{a-1}{p-1} A\rt)^{\frac{p-1}{a-1}}\notag\\
&\leq A + \lt(\frac{p-1}{(a-p)A^{\frac1a}}\rt)^{\frac{a(p-1)}{a-p}} \lt(\int_{\R^n_+} \|\na f\|_*^p dz\rt)^{\frac{a-1}{a-p}},
\end{align}
with equality if $f =h_p$. Combining \eqref{eq:Young2} with \eqref{eq:nonhomo1} yields
\begin{align}\label{eq:nonhomo2}
\int_{\pa \R^n_+} |f|^{\frac{p(a-1)}{a-p}} dx &\leq (a-n) \int_{\R^n_+} |f|^{\frac{p(a-1)}{a-p}} dz + \lt(\frac{p-1}{(a-p)A^{\frac1a}}\rt)^{\frac{a(p-1)}{a-p}} \lt(\int_{\R^n_+} \|\na f\|_*^p dz\rt)^{\frac{a-1}{a-p}},
\end{align}
with equality if $f =h_p$. If $a=n$, the inequality \eqref{eq:nonhomo2} gives the sharp Sobolev trace inequalities due to Nazaret \cite{Nazaret2006}. If $a>n$, denote $B = \lt(\frac{p-1}{(a-p)A^{\frac1a}}\rt)^{\frac{a(p-1)}{a-p}}$ and $f_\lam = \lam^{\frac{n(a-p)}{ap}} f(\lam\,\cdot )$, for $\lam >0$. It is obvious that 
\[
\int_{\R^n_+} |f_\lam|^{\frac{ap}{a-p}} dz = \int_{\R^n_+} |f|^{\frac{ap}{a-p}} dz =\int_{\R^n_+} h_p^{\frac{ap}{a-p}} dz,
\]
\[
\int_{\pa \R^n_+} |f_\lam|^{\frac{p(a-1)}{a-p}} dx = \lam^{1-\frac{n}a}\int_{\pa \R^n_+} |f|^{\frac{p(a-1)}{a-p}} dx,
\]
\[
\int_{\R^n_+} |f_\lam|^{\frac{p(a-1)}{a-p}} dz = \lam^{-\frac na} \int_{\R^n_+} |f|^{\frac{p(a-1)}{a-p}} dz,
\]
and
\[
\int_{\R^n_+} \|\na f_\lam\|_*^p dz = \lambda^{p\lt(1-\frac na\rt)} \int_{\R^n_+} \|\na f\|_*^p dz.
\]
Replacing $f$ by $f_\lam$ in \eqref{eq:nonhomo2}, we get
\begin{align}\label{eq:lambda}
\int_{\pa \R^n_+} |f|^{\frac{p(a-1)}{a-p}} dx &\leq \lam^{-1}(a-n) \int_{\R^n_+} |f(z)|^{\frac{p(a-1)}{a-p}} dz \notag\\
&\quad + \lam^{\frac{(a-n)(p-1)}{a-p}}\lt(\frac{p-1}{(a-p)A^{\frac1a}}\rt)^{\frac{a(p-1)}{a-p}} \lt(\int_{\R^n_+} \|\na f\|_*^p dz\rt)^{\frac{a-1}{a-p}},
\end{align}
for any $\lambda >0$. The estimate \eqref{eq:lambda} is exactly formula $(46)$ in \cite{BCFGG17} with a different normalization condition on $f$. Optimizing the right hand side of \eqref{eq:lambda} over $\lambda >0$, we obtain the inequality \eqref{eq:GNtrace} for any function $f$ with $\|f\|_{\frac{ap}{a-p}} = \|h_p\|_{\frac{ap}{a-p}}$. By homogeneity, the inequality \eqref{eq:GNtrace} holds for any function $f$.

The optimality of \eqref{eq:GNtrace} is checked directly by using function $h_p$. Indeed, the equality holds true in \eqref{eq:nonhomo2} if $f =h_p$ and $\lambda =1$. Note that
\[
\int_{\R^n_+} \|\na h_p\|_*^p dz = \frac{(a-p)^p}{(p-1)^p} A,
\]
hence the right hand side of \eqref{eq:nonhomo2} is minimized at $\lam =1$ when $f =h_p$. This proves the optimality of \eqref{eq:GNtrace}. The proof of Theorem \ref{Maintheorem1} was completely finished.

\section{Proof of Theorem \ref{Maintheorem2}}
This section is devoted to prove Theorem \ref{Maintheorem2} concerning to the extremal functions for the sharp Gagliardo--Nirenberg trace inequality \eqref{eq:GNtrace}. We follow the arguments in \cite{CNV2004} dealing with the identification of the extremal functions for the sharp Sobolev inequality. To do this, we trace back to the equality cases in the estimates used in the previous section on the Monge--Amp\`ere equation \eqref{eq:MAequation}, the definition of mass transportation, the arithmetic--geometric inequality \eqref{eq:AGinequality}, H\"older inequality \eqref{eq:YoungBre}, the Young inequalities \eqref{eq:Young1} and \eqref{eq:Young2}, and the integration by parts formula \eqref{eq:IBPBV0}. We remark that the integration by parts formula \eqref{eq:IBPBV0} is valid under the extra assumptions on $f$ and $g$ which ensure the regularity of the mass transport. If $g$ is only in $L^{\frac{ap}{a-p}}(\R^n_+)$ and is not compactly supported, then the normal derivative of $\vphi$ has no reason even to exist on the boundary. So the integration by parts formula \eqref{eq:IBPBV0} does not exists in this case. The first step in determining the optimal functions for the Gagliardo--Nirenberg trace inequality \eqref{eq:GNtrace} is to establish an integration by parts inequality which holds for more general $f$ and $g$ without extra assumptions on the smoothness or the compactly supported property. This is the content of the following lemma. 

\begin{lemma}\label{IBPinequality}
Let $f,g \in \mathcal D^{p, \frac{p(a-1)}{a-p}}(\R^n_+)$ be nonnegative function such that $\|f\|_{\frac{ap}{a-p}} = \|g\|_{\frac{ap}{a-p}}$ and $\int_{\R^n_+} \|z+e\|^{\frac p{p-1}} g^{\frac{ap}{a-p}} dz < \infty$. Let $\na\vphi$ be the Brenier map pushing $f^{\frac{ap}{a-p}} dz$ forward to $g^{\frac{ap}{a-p}} dz$. Then, it holds
\begin{equation}\label{eq:IBPinequality}
\int_{\R^n_+} f(z)^{\frac{p(a-1)}{a-p}}\, \Delta_A\psi(z) dz + \int_{\pa \R^n_+} f(x)^{\frac{p(a-1)}{a-p}} dx \leq -\frac{p(a-1)}{a-p} \int_{\R^n_+} \na f \cdot \na\psi f^{\frac{p(a-1)}{a-p}} dz,
\end{equation}
here $\psi(z) = \phi(z) + z\cdot e$.
\end{lemma}
\begin{proof}
Since $\int_{\R^n_+} g^{\frac{ap}{a-p}} \|z+e\|^{\frac p{p-1}} dz < \infty$, by definition of the mass transport, we have 
\begin{equation}\label{eq:brefi}
\int_{\R^n_+} f(z)^{\frac{ap}{a-p}} \|\na\psi(z)\|^{\frac p{p-1}} dz < \infty.
\end{equation}
Let $\Om$ be the interior of $\{z\in \R^n: \vphi(z) < \infty\}$, then the support of $f$ is contained in $\o \Om$, and $\pa \Om$ is of zero measure. The function $\vphi$ is convex, hence it is differentiable a.e. in $\Om$. By Fubini theorem, for a.e. $t>0$, $\na \vphi$ exists a.e. in $\Om(t) = \Om \cap (\{t\} \times \R^{n-1})$ with respect to $(n-1)$-dimensional Lebesgue measure. Let $t_0 >0$ be such a number, and fix $z_0 = (t_0,x_0) \in \Om \cap \R^n_+$. Taking $k_0$ such that $1/k_0 < t_0$, and for $k\geq k_0$ we define $f_k(z) = f(z)\theta(k t)$, $z =(t,x)$, where $\theta$ is a smooth, increasing function on $[0,\infty)$ such that $0\leq \theta\leq 1$, $\theta =0$ on $[0,1]$ and $\theta =1$ on $[2,\infty)$. Thus, the support of $f_k$ is contained in $\{z =(t,x): t \geq 1/k\}$. For $\ep >0$ small enough ($\ep \ll t_0 -1/k_0$), we define
\[
f_{k,\ep}(z) = \min\lt\{f_k\lt(z_0 + \frac{z-z_0}{1-\ep}\rt), \, f_k(z) \chi(\ep z)\rt\},
\]
where $\chi$ is a $C^\infty$ cut-off function, $\chi(z) =1$ if $|z|\leq 1/2$ and $\chi(z) =0$ if $|z|\geq 1$ and $0 \leq \chi\leq 1$. For $\de >0$, denote $f_{k,\ep,\de} = f_{k,\ep} \ast \eta_\de$, where $\eta_\de = \de^{-n} \eta(\de^{-1}\,\cdot)$ with $\eta\in C_0^\infty(\R^n)$, $\eta \geq 0$ and $\int_{\R^n} \eta(z) dz =1$. For $\de$ small enough ($\de$ is smaller than the distance from the support of $f_{k,\ep}$ to $\pa \Om$), $f_{k,\ep,\de}$ is compactly supported in $\Om$ and smooth, i.e., $f_{k,\ep,\de} \in C_0^\infty(\Om)$.
Since $\Delta_A \psi \leq \Delta_{\mathcal D'} \psi$, and $\na \vphi$ has locally bounded variation in $\Om$, applying integration by parts for functions of bounded variation, we have
\[
\int_{\{t\geq t_0\}} f_{k,\ep,\de}^{\frac{p(a-1)}{a-p}} \Delta_A\psi dz \leq -\frac{p(a-1)}{a-p}\int_{\{t\geq t_0\}} \na f_{k,\ep,\de} \cdot \na \psi f_{k,\ep,\de}^{a\frac{p-1}{a-p}} dz -\int_{\{t=t_0\}} f_{k,\ep,\de}^{\frac{p(a-1)}{a-p}} \na \psi \cdot e dx.
\]
Since $\na \vphi \in \R^n_+$ a.e., then $\na\vphi \cdot e \geq 0$ which together with the previous estimate implies
\[
\int_{\{t\geq t_0\}} f_{k,\ep,\de}^{\frac{p(a-1)}{a-p}} \Delta_A\psi dz \leq  -\frac{p(a-1)}{a-p}\int_{\{t\geq t_0\}} \na f_{k,\ep,\de} \cdot \na \psi f_{k,\ep,\de}^{a\frac{p-1}{a-p}} dz -\int_{\{t=t_0\}} f_{k,\ep,\de}^{\frac{p(a-1)}{a-p}} dx,
\]
or equivalently
\begin{equation}\label{eq:IBP}
\int_{\{t\geq t_0\}} f_{k,\ep,\de}^{\frac{p(a-1)}{a-p}} \Delta_A\psi dz + \int_{\{t=t_0\}} f_{k,\ep,\de}^{\frac{p(a-1)}{a-p}} dx \leq -\frac{p(a-1)}{a-p}\int_{\{t\geq t_0\}} \na f_{k,\ep,\de} \cdot \na \psi f_{k,\ep,\de}^{a\frac{p-1}{a-p}} dz.
\end{equation}
For $\de$ is sufficient small, the support of $f_{k,\ep,\de}$ is contained in $\Om' \subset\subset \Om$, hence $\na\psi$ is bounded uniformly on the support of $f_{k,\ep,\de}$. Moreover, $f_{k,\ep,\de} \to f_{k,\ep}$ in $\mathcal D^{p,\frac{p(a-1)}{a-p}}(\R^n_+)$, by Gagliardo--Nirenberg trace inequality, we get $f_{k,\ep,\de} \to f_{k,\ep}$ in $L^{\frac{p(a-1)}{a-p}}(\{t_0\}\times \R^{n-1})$ when $\de \to 0$. Letting $\de \to 0$ in \eqref{eq:IBP} and using Fatou's lemma and $\De_A \psi =\De_A \vphi \geq 0$, we get
\begin{equation}\label{eq:IBP1}
\int_{\{t\geq t_0\}} f_{k,\ep}^{\frac{p(a-1)}{a-p}} \Delta_A\psi dz + \int_{\{t=t_0\}} f_{k,\ep}^{\frac{p(a-1)}{a-p}} dx \leq -\frac{p(a-1)}{a-p}\int_{\{t\geq t_0\}} \na f_{k,\ep} \cdot \na \psi f_{k,\ep}^{a\frac{p-1}{a-p}} dz.
\end{equation}
It remains to pass $\ep \to 0$ in \eqref{eq:IBP1}. Repeating the argument in the proof of Lemma $7$ in \cite{CNV2004} and using Gagliardo--Nirenberg trace inequalities together with \eqref{eq:brefi}, we obtain
\begin{equation}\label{eq:IBP2}
\int_{\{t\geq t_0\}} f_{k}^{\frac{p(a-1)}{a-p}} \Delta_A\psi dz + \int_{\{t=t_0\}} f_{k}^{\frac{p(a-1)}{a-p}} dx \leq -\frac{p(a-1)}{a-p}\int_{\{t\geq t_0\}} \na f_{k} \cdot \na \psi f_{k}^{a\frac{p-1}{a-p}} dz.
\end{equation}
Since $\na f_k(t,x) = \theta(kt) \na f(t,x) + f(t,x) k\theta'(kt)$ and $\theta$ is increasing, then
\[
-\na f_k(t,x) \cdot \na \psi \leq -\theta(tk) \na f(t,x) \cdot \na \psi(t,x).
\]
Letting $k\to\infty$ in \eqref{eq:IBP2}, we get
\begin{equation}\label{eq:IBP3}
\int_{\{t\geq t_0\}} f^{\frac{p(a-1)}{a-p}} \Delta_A\psi dz + \int_{\{t=t_0\}} f^{\frac{p(a-1)}{a-p}} dx \leq -\frac{p(a-1)}{a-p}\int_{\{t\geq t_0\}} \na f \cdot \na \psi f^{a\frac{p-1}{a-p}} dz.
\end{equation}
It remains to pass $t_0\to 0$ in \eqref{eq:IBP3}. By Fatou's lemma and the fact $\na f \cdot \na \psi f^{a\frac{p-1}{a-p}} \in L^1(\R^n_+)$ by \eqref{eq:brefi}, it is enough to show that
\begin{equation}\label{eq:suffcond}
\lim_{t_0 \to 0}\int_{\{t=t_0\}} f^{\frac{p(a-1)}{a-p}} dx = \int_{\pa \R^n_+} f^{\frac{p(a-1)}{a-p}} dx.
\end{equation}
For $t >0$, define the function $h_t$ on $\R^n_+$ by $h_t(z) =f(z+te) -f(z)$. It is not hard to see that
\[
\int_{\R^n_+} \|\na h_t\|_*^p dz = \int_{\R^n_+} \|\na f(z+te) -\na f(z)\|_*^p dz \to 0,
\]
and
\[
\int_{\R^n_+} |h_t(z)|^{\frac{p(a-1)}{a-p}} dz = \int_{\R^n_+} |f(z+te) -f(z)|^{\frac{p(a-1)}{a-p}} dz \to 0,
\]
as $t\to 0^+$ by the continuity of $L^p-$norm under the translation. Applying Gagliardo--Nirenberg trace inequality, we get
\[
\lim_{t\to 0} \int_{\pa \R^n_+} |f(t,x) -f(x)|^{\frac{p(a-1)}{a-p}} dx =\lim_{t\to 0} \int_{\pa \R^n_+} |h_t|^{\frac{p(a-1)}{a-p}} dx =0.
\]
This implies \eqref{eq:suffcond}.
\end{proof}

Suppose that equality holds true in $\eqref{eq:GNtrace}$ for a function $f\not =0$. It is a standard argument (by writing $f$ as $f = f^+ -f^-$ and using convexity) that $f$ does not change sign, hence without loss of generality, we can assume that $f$ is nonnegative. Note that $\mathcal D^{p,\frac{p(a-1)}{a-p}}(\R^n_+) \subset L^{\frac{ap}{a-p}}(\R^n_+)$, by the homegeneity, we can assume that
\[
\int_{\R^n_+} f^{\frac{ap}{a-p}} dz = \int_{\R^n_+} h_p^{\frac{ap}{a-p}} dz.
\]
Let $\lambda_0$ be the number minimizing the right hand side of \eqref{eq:lambda}. Then \eqref{eq:nonhomo2} becomes equality for $f_{\lambda_0} = \lambda_0^{\frac{n(a-p)}{ap}} f(\lambda_0\, \cdot)$. Let $\na \vphi$ be the Brenier map pushing $f_{\lam_0}^{\frac{ap}{a-p}} dz$ forward to $h_p^{\frac{ap}{a-p}}dz$, and $\Om =\{z\in \R^n: \vphi(z) < \infty\}$. Recall that \eqref{eq:AlekLap} holds in general without any extra assumptions on the regularity of $f$ and $g$. Applying \eqref{eq:AlekLap} and \eqref{eq:IBPinequality} for $f_{\lam_0}$ and $h_p$, we see that \eqref{eq:IBPBV*} still holds for $f_{\lam_0}$ and $h_p$ which is equivalent to 
\begin{align}\label{eq:IBPequality}
\int_{\pa \R^n_+} f_{\lam_0}^{\frac{p(a-1)}{a-p}} dx &\leq (a-n) \int_{\R^n_+} f_{\lam_0}(z)^{\frac{p(a-1)}{a-p}} dz -\frac{p(a-1)}{a-p} \int_{\R^n_+} \na f_{\lam_0} \cdot \na \psi f_{\lam_0}^{\frac{a(p-1)}{a-p}} dz\notag\\
&\qquad\qquad -a\int_{\R^n_+} h_p^{\frac{p(a-1)}{a-p}} dz.
\end{align}
By H\"older inequality and Young inequality, \eqref{eq:IBPequality} leads to the inequalities \eqref{eq:IBPBV**} and \eqref{eq:nonhomo2} for $f_{\lam_0}$ and $h_p$, i.e.,
\begin{align}\label{eq:IBPBV***}
\int_{\pa \R^n_+} f_{\lam_0}^{\frac{p(a-1)}{a-p}} dx &\leq  -a\int_{\R^n_+} h_p^{\frac{p(a-1)}{a-p}} dz + (a-n) \int_{\R^n_+} f_{\lam_0}(z)^{\frac{p(a-1)}{a-p}} dz \notag\\
&\qquad + \frac{p(a-1)}{a-p}\lt(\int_{\R^n_+} \|\na f_{\lam_0}\|_*^p dz\rt)^{\frac1p}\lt( \int_{\R^n_+} \|z+e\|^{\frac p{p-1}} h_p^{\frac {ap}{a-p}} dz\rt)^{\frac{p-1}p}\notag\\
&= -a A + (a-n) \int_{\R^n_+} f_{\lam_0}^{\frac{p(a-1)}{a-p}} dz + \frac{p(a-1)}{a-p} A^{\frac{p-1}p}\lt(\int_{\R^n_+} \|\na f_{\lam_0}\|_*^p dz\rt)^{\frac1p},
\end{align}
and
\begin{align}\label{eq:nonhomo2*}
\int_{\pa \R^n_+} f_{\lam_0}^{\frac{p(a-1)}{a-p}} dx &\leq (a-n) \int_{\R^n_+} f_{\lam_0}(z)^{\frac{p(a-1)}{a-p}} dz \notag\\
&\qquad\qquad  + \lt(\frac{p-1}{(a-p)A^{\frac1a}}\rt)^{\frac{a(p-1)}{a-p}} \lt(\int_{\R^n_+} \|\na f_{\lam_0}\|_*^p dz\rt)^{\frac{a-1}{a-p}}.
\end{align}
As discussed above, we have equality in \eqref{eq:nonhomo2*}. Hence equality holds in \eqref{eq:IBPBV***}. In particular, we must have equalities in H\"older inequality \eqref{eq:YoungBre} applied to $f_{\lam_0}$ and in the arithmetic--geometric inequality \eqref{eq:AGinequality}. Equality in H\"older inequality implies
\[
\|\na f_{\lam_0}(z)\|_*^p = k \|\na \psi(z)\|^{\frac p{p-1}} f_{\lam_0}(z)^{\frac{ap}{a-p}},
\]
for almost every $z\in \Om \cap \R^n_+$. This equality and the argument in \cite{CNV2004} shows that $f_{\lam_0}$ is positive on $\Om\cap \R^n_+$ and the support of $f_{\lam_0}$ is $\o\Om \cap \o{\R^n_+}$.

Using again the argument in \cite{CNV2004} we conclude that $D^2_{\mathcal D'}\vphi$ has no singular part in $\Om \cap \R^n_+$, i.e., $D^2_{\mathcal D'}\vphi$ is absolutely continuous with respect Lebesgue measure on $\Om\cap \R^n_+$.

If $a > n$, the equality in the arithmetic--geometric inequality implies that $D_{\mathcal D'}^2 \vphi$ is identity matrix almost everywhere in $\Om \cap \R^n_+$, hence $\na\vphi(t,x) = (t-t_0,x-x_0)$ for some $(t_0,x_0)\in \R^n$. Since $\na\vphi$ sends the interior of support of $f_{\lam_0}$ to the interior of support of $h_p$ (which is $\R^n_+$), then the interior of support of $f_{\lam_0}$ is $\R^n_+ + t_0 e$ which implies $t_0 \geq 0$. Therefore, we have $f_{\lam_0}(z) = 1_{\R^n_++t_0 e}(z) h_p(z -(t_0,x_0))$ which forces $t_0 =0$. This means that $f(z) = \lam_0^{-\frac{n(a-p)}{ap}} h_p(\lam_0^{-1}z -(0,x_0))$ as claimed.

If $a=n$, the equality in the arithmetic--geometric inequality implies that $D_{\mathcal D'}^2 \vphi$ is proportional to the identity matrix at almost everywhere in $\Om\cap \R^n_+$. Using regularizing process as done in \cite{CNV2004}, we conclude that $D_{\mathcal D'}^2\vphi$ is a constant multiple of identity matrix in $\Om\cap \R^n_+$. Hence $\na\vphi(t,x) = \lam (t-t_0,x-x_0)$ for some $\lam >0$, $(t_0,x_0)\in \R^n$. Arguing as in the case $a >n$, we get $t_0= 0$, which implies $f(z) = \lt(\frac{\lam}{\lam_0}\rt)^{\frac{n(a-p)}{ap}} h_p(\frac{\lam}{\lam_0}z -(0,\lam x_0))$ as wanted.

\section{Proof of the sharp affine Gagliardo--Nirenberg trace inequalities: Proof of Theorem \ref{affineGNtrace}}
Before proving Theorem \ref{affineGNtrace}, let us recall some useful facts from theory of convex bodies. The book of Schneider \cite{Schneider} is a good reference for this purpose. A subset $K$ of $\R^n$ is called a convex body if it is a compact convex subset of $\R^n$ with non-empty interior. A convex body $K$ is symmetric if $K = -K$. For a convex body $K \subset \R^n$, its support function $h_K$ is defined by
\[
h_K(x) = \max\{\la x,y\ra\, :\, y \in K\}.
\]
It is well-known that a convex body is completely determined by its support function. If $K$ is a symmetric convex body, then its Minkowski functional (or the gauge) $\|\cdot\|_K$ is defined by
\[
\|x\|_K = \inf\{\lam >0\, :\, x \in \lam K\}.
\]
$\|\cdot\|_K$ is a norm in $\R^n$. The polar body $K^\circ$ of $K$ is defined by
\[
K^\circ = \{x\in \R^n\, :\, \la x, y\ra \leq 1, \quad\forall y\in K\}.
\]
It is easy to see that $h_K = \|\cdot\|_{K^\circ}^{-1}$. Moreover, the following formula holds
\begin{equation}\label{eq:volume}
{\rm vol_n}(K) = \frac1n \int_{S^{n-1}} \|y\|_K^{-n} dy,
\end{equation}
here ${\rm vol_k}$ denotes the volume in $\R^k$. For a convex body $K$ in $\R^n$ containing the origin in its interior, its $L_p$ centroid  body $\Gamma_p K$, $p\geq 1$ is defined by
\[
h_{\Gamma_p K}(x)^p = \frac1{a_{n,p} {\rm vol}(K)} \int_K |\la x,y\ra|^p dy,
\]
for $x \in \R^n$, where $a_{n,p} = \frac{\om_{n+p}}{\om_2 \om_n \om_{p-1}}$ is a normalized constant such that $\Gamma_p B_2^n = B_2^n$ with $B_2^n$ denoting the unit ball of $\R^n$. The $L_p$ Busemann--Petty centroid inequality proved by Lutwak, Yang and Zhang (see \cite{LYZ2000}) states that for any convex body $K$ in $\R^n$ containing the origin in its interior we have
\begin{equation}\label{eq:Lpcentroid}
{\rm vol_n}(\Gamma_p K) \geq {\rm vol_n}(K),
\end{equation}
with equality if and only if $K$ is an origin--symmetric ellipsoid. We refer the interest reader to \cite{HS09Iso} for a generalization of the inequality \eqref{eq:Lpcentroid}.

Let $f \in \mathcal D^p(\R^n_+)$ which is not identical $0$. We denote by $B_p(f)$ the unit ball in $\R^{n-1}$ with respect to the norm
\[
\|u\|_{B_p(f)} = \lt(\int_{\R^n_+} |\la u, \na_x f(t,x)\ra|^p dxdt\rt)^{\frac1p}.
\]
From the definition of $\mathcal E_p(f)$ and the formula \eqref{eq:volume}, we get
\begin{equation}\label{eq:volumea}
\mathcal E_p(f) = c_{n-1,p} \lt((n-1) {\rm vol_{n-1}}(B_p(f))\rt)^{-\frac1{n-1}}.
\end{equation}
Let $K_p(f)$ be a convex body in $\R^{n-1}$ whose support function is
\[
h_{K_p(f)}(u) = \lt(\int_{S^{n-2}} \|v\|_{B_p(f)}^{-n-p+1} |\la u, v\ra|^p dv\rt)^{\frac1p}.
\]
Following the proof of Lemmas $3.2$ and $3.3$ in \cite{Nguyen2016}, we have the following identities.
\begin{lemma}\label{relation1}
It holds
\[
K_p(f) = \lt((n+p-1) \al_{n-1,p} {\rm vol_{n-1}}(B_p(f))\rt)^{\frac1p} \Gamma_p B_p(f),
\]
\[
\int_{\R^n_+} h_{K_p(f)}(\na_x f(t,x))^p dx dt  = \lt(\frac{\mathcal E_p(f)}{c_{n-1,p}}\rt)^{1-n}.
\]
\end{lemma}
It follows from the $L_p$ Busemann--Petty centroid inequality \eqref{eq:Lpcentroid} that
\[
{\rm vol_{n-1}}(\Gamma_p B_p(f)) \geq {\rm vol_{n-1}}(B_p(f)),
\]
with equality if and only if $B_p(f)$ is an origin--symmetric ellipsoid in $\R^{n-1}$. Combining the preceding inequality together with Lemma \ref{relation1}, we get
\begin{equation}\label{eq:aaaa}
\lt(\frac{\om_{n-1}}{{\rm vol_{n-1}}(K_p(f))}\rt)^{\frac1{n-1}} \lt(\int_{\R^n_+} h_{K_p(f)}(\na_x f(t,x))^p dx dt \rt)^{\frac1p} \leq \mathcal E_p(f),
\end{equation}
with equality if and only if $B_p(f)$ is an origin--symmetric ellipsoid in $\R^{n-1}$. Denote 
\[
\overline K_p(f) =\lt(\frac{\om_{n-1}}{{\rm vol_{n-1}}(K_p(f))}\rt)^{\frac1{n-1}} K_p(f).
\]
Then ${\rm vol_{n-1}}(\overline K_p(f)) = \om_{n-1}$ and the inequality \eqref{eq:aaaa} is equivalent to
\begin{equation}\label{eq:aaaa*}
\int_{\R^n_+} h_{\overline K_p(f)}(\na_x f(t,x))^p dx dt \leq \mathcal E_p(f)^p,
\end{equation}
with equality if and only if $B_p(f)$ is an origin--symmetric ellipsoid in $\R^{n-1}$. Given a vector $\vec{a} \in \R^{n-1}$, let $B_{p,\vec{a}}(f)$ denote the unit ball of the norm
\[
\|(t,x)\|_{p,\vec{a},f} = \lt(|t|^q + \|x + t \vec{a}\|_{\overline K_p(f)}^q\rt)^{\frac1q},\qquad q =\frac p{p-1}.
\]
We can readily check that
\begin{equation}\label{eq:volumeBpf}
{\rm vol_{n-1}} (B_{p,\vec{a}}(f)) = \frac2 q \frac{\Gamma(\frac{n-1}q +1) \Gamma(\frac1q)}{\Gamma(\frac nq +1)} \,\om_{n-1},
\end{equation}
and 
\begin{equation}\label{eq:dualnorm}
h_{B_{p,\vec{a}}(f)}((t,x)) = \sup_{\|(s,y)\|_{p,f} \leq 1} \lt(ts + \la x, y\ra\rt) = \lt(h_{\overline K_p(f)}(x)^p + |t-\la\vec{a}, x\ra|^p\rt)^{\frac1p}.
\end{equation}
Hence, it follows from \eqref{eq:aaaa*} and \eqref{eq:dualnorm} that
\begin{equation}\label{eq:bbbb}
\mathcal E_p(f)^p + \int_{\R^n_+} \lt|D_{(1,-\vec{a})}f\rt|^p dx dt \geq \int_{\R^n_+} h_{B_{p,\vec{a}}(f)}(\na f(t,x))^p dx dt.
\end{equation}
Now, applying the sharp Gagliardo--Nirenberg trace inequality \eqref{eq:GNtrace} with the norm $\|\cdot\|_{p,\vec{a},f}$, we obtain
\begin{equation}\label{eq:specialGN}
\|g\|_{L^{\frac{p(a-1)}{a-p}}(\pa \R^n_+)} \leq \mathcal D_{n,a,p} \lt(\int_{\R^n_+} h_{B_{p,\vec{a}}(f)}(\na g)^p dz\rt)^{\frac{\theta}p} \|g\|_{L^{\frac{p(a-1)}{a-p}}(\R^n_+)}^{(1-\theta)}
\end{equation}
where $\theta$ is given by \eqref{eq:theta} and $\mathcal D_{n,a,p}$ is given by \eqref{eq:sharpconstantaffineGN}. By Theorem \ref{Maintheorem2}, there is equality in \eqref{eq:specialGN} if and only if  
\begin{equation}\label{eq:equalityg}
g(t,x) = \lt((1+ t)^q + \|x + t \vec{a}\|_{\overline K_p(f)}^q \rt)^{-\frac{a-p}{p}},
\end{equation}
up to a translation on $\{0\} \times \R^{n-1}$, a dilation and a multiplicative constant.

Combining \eqref{eq:bbbb} and \eqref{eq:specialGN} together, we get
\begin{equation}\label{eq:1st}
\|f\|_{L^{\frac{p(a-1)}{a-p}}(\pa \R^n_+)} \leq \mathcal D_{n,a,p} \lt(\mathcal E_p(f)^p + \int_{\R^n_+} \lt|D_{(1,-\vec{a})}f\rt|^p dx dt\rt)^{\frac{\theta}p}\|f\|_{L^{\frac{p(a-1)}{a-p}}(\R^n_+)}^{1-\theta}.
\end{equation}
Since the function
\[
\vec{a} \mapsto \int_{\R^n_+} \lt|D_{(1,-\vec{a})}f\rt|^p dx dt
\]
is strictly convex on $\R^{n-1}$, then there exists uniquely $\vec{a}_0$ such that
\begin{equation}\label{eq:minfollowa}
\int_{\R^n_+} \lt|D_{(1,-\vec{a}_0)}f\rt|^p dx dt = \min_{\vec{a} \in \R^{n-1}} \int_{\R^n_+} \lt|D_{(1,-\vec{a})}f\rt|^p dx dt.
\end{equation}
Consequently, we get from \eqref{eq:1st} that
\begin{equation}\label{eq:2nd}
\|f\|_{L^{\frac{p(a-1)}{a-p}}(\pa \R^n_+)} \leq \mathcal D_{n,a,p} \lt(\mathcal E_p(f)^p + \min_{\vec{a}\in \R^{n-1}}\int_{\R^n_+} \lt|D_{(1,-\vec{a})}f\rt|^p dx dt\rt)^{\frac{\theta}p}\|f\|_{L^{\frac{p(a-1)}{a-p}}(\R^n_+)}^{1-\theta}.
\end{equation}
For any $\lam >0$, applying \eqref{eq:2nd} to function $f_\lam(t,x) = f(\lam t,x)$, we get
\begin{align}\label{eq:3th}
\|f\|_{L^{\frac{p(a-1)}{a-p}}(\pa \R^n_+)} &\leq \mathcal D_{n,a,p} \Bigg(\lam^{-1 -\frac{1-\theta}\theta \frac{a-p}{a-1}}\mathcal E_p(f)^p  + \lam^{p-1 -\frac{1-\theta}\theta \frac{a-p}{a-1}} \min_{\vec{a}\in \R^{n-1}}\int_{\R^n_+} \lt|D_{(1,-\vec{a})}f\rt|^p dx dt\Bigg)^{\frac{\theta}p}\notag\\
&\qquad\qquad  \times \|f\|_{L^{\frac{p(a-1)}{a-p}}(\R^n_+)}^{1-\theta}.
\end{align}
Note that $p-1 -\frac{1-\theta}\theta \frac{a-p}{a-1} >0$ by the assumption \eqref{eq:theta}, hence the right hand side of \eqref{eq:3th} is minimized by 
\begin{equation}\label{eq:lambdaminimize}
\lam_0 = \lt(\frac{p(a-n) + n-1}{(p-1)(n-1)}\rt)^{\frac1p} \frac{\mathcal E_p(f)}{\lt(\min\limits_{\vec{a}\in \R^{n-1}}\int_{\R^n_+} \lt|D_{(1,-\vec{a})}f\rt|^p dx dt\rt)^{\frac1p}}.
\end{equation}
Taking $\lam =\lam_0$ in \eqref{eq:3th}, we obtain the inequality \eqref{eq:affineGNtrace} as desired. 

It remains to check the condition of equalities in \eqref{eq:affineGNtrace}. If $f =h_p$ then $\overline K_p(f)$ is an Euclidean unit ball in $\R^{n-1}$. By the convexity and the evenness of $h_p$ in $x$ variables, the minimum in \eqref{eq:minfollowa} is attained by $\vec{a}_0 = 0$, hence $\|(t,x)\|_{p,0,h_p} = (|t|^q + |x|^q)^{\frac1q}$, $q =p/(p-1)$. Thus, we have equality in \eqref{eq:aaaa*}, \eqref{eq:bbbb} and \eqref{eq:1st} with $f =h_p$ and $\vec{a} =0$ which imply  the equality in \eqref{eq:2nd} for $f =h_p$. It is easy to check that $\lam_0$ defined by \eqref{eq:lambdaminimize} is equal to $1$ if $f =h_p$. Hence, the equality in \eqref{eq:3th} for $f =h_p$ and $\lam =1$ proves the equality in \eqref{eq:affineGNtrace} for $f =h_p$. The ${\rm SGL}_{n,+}$ invariant property of the inequality \eqref{eq:affineGNtrace} implies that equality holds in \eqref{eq:affineGNtrace} for functions of the form \eqref{eq:extremalGN}.

Conversely, suppose that equality holds in \eqref{eq:affineGNtrace} for a non-zero function $f$. Let $\vec{a}_0$ denote the unique minimizer in \eqref{eq:minfollowa} and let $\lam_0$ be defined by \eqref{eq:lambdaminimize}. Consider the function $f_{\lam_0}$ defined by $f_{\lam_0}(t,x) = f(\lam_0 t,x)$. Note that $f_{\lam_0} = f$ on $\pa \R^n_+$. Applying the inequalities \eqref{eq:aaaa*}, \eqref{eq:bbbb} and \eqref{eq:specialGN} for $f_{\lam_0}$ and $\vec{b}_0 := \lam_0 \vec{a}_0$, we get
\begin{align*}
\|f\|_{L^{\frac{p(a-1)}{a-p}}(\pa \R^n_+)} & = \|f_{\lam_0}\|_{L^{\frac{p(a-1)}{a-p}}(\pa \R^n_+)}\\
&\leq \mathcal D_{n,a,p} \lt(\int_{\R^n_+} h_{B_{p,\vec{b}_0}(f_{\lam_0})}(\na f_{\lam_0})^p dz\rt)^{\frac{\theta}p} \|f_{\lam_0}\|_{L^{\frac{p(a-1)}{a-p}}(\R^n_+)}^{(1-\theta)}\\
&= \mathcal D_{n,a,p} \lt(\int_{\R^n_+} (h_{\overline K_p(f_{\lam_0})}(\na_x f_{\lam_0})^p + |D_{(1,-\vec{b}_0)} f_{\lam_0}|^p) dx dt \rt)^{\frac{\theta}p} \|f_{\lam_0}\|_{L^{\frac{p(a-1)}{a-p}}(\R^n_+)}^{(1-\theta)}\\
&\leq \mathcal D_{n,a,p}\lt(\mathcal E_p(f_{\lam_0})^p + \lam_0^{p-1}\int_{\R^n_+}|D_{(1,-\vec{a}_0)} f|^p dx dt\rt)^{\frac\theta p}\|f_{\lam_0}\|_{L^{\frac{p(a-1)}{a-p}}(\R^n_+)}^{(1-\theta)}\\
&=\mathcal D_{n,a,p} \Bigg(\lam_0^{-1 -\frac{1-\theta}\theta \frac{a-p}{a-1}}\mathcal E_p(f)^p  + \lam_0^{p-1 -\frac{1-\theta}\theta \frac{a-p}{a-1}} \int_{\R^n_+} \lt|D_{(1,-\vec{a}_0)}f\rt|^p dx dt\Bigg)^{\frac{\theta}p}\\
&\qquad\qquad  \times \|f\|_{L^{\frac{p(a-1)}{a-p}}(\R^n_+)}^{1-\theta}\\
&=\mathcal D_{n,a,p}\,  \mathcal A_p(f)^{\theta} \|f\|_{L^{\frac{p(a-1)}{a-p}}(\R^n_+)}^{1-\theta}\\
&= \|f\|_{L^{\frac{p(a-1)}{a-p}}(\pa \R^n_+)}.
\end{align*}
As a consequence, all estimates for $f_{\lam_0}$ above must be equalities. Firstly, we must have equality in the Gagliardo--Nirenberg trace inequality \eqref{eq:specialGN} for functions $f$ and $g$ replaced by function $f_{\lam_0}$ and the vector $\vec{a}$ replaced by vector $\vec{b}_0$. Theorem \ref{Maintheorem2} implies the existence of  $c\not =0$, $\mu >0$ and $x_0\in \R^{n-1}$ such that
\[
f_{\lam_0}(t,x) = c \, \Big{\|}e + \frac{(t, x-x_0)}\mu \Big{\|}^{-\frac{a-p}{p-1}}_{p, \vec{b}_0, f_{\lam_0}} =c\, \lt(\lt(1+ \frac t\mu\rt)^q + \Big{\|} \frac{x-x_0 + t\vec{b}_0}\mu\Big{\|}_{\overline K_p(f_{\lam_0})}^q \rt)^{-\frac{a-p}p}.
\]
Secondly, we have the equality holds in \eqref{eq:aaaa*} for $f_{\lam_0}$ which implies that $B_p(f_{\lam_0})$ is an origin--symmetric ellipsoid. Then so is $K_p(f_{\lam_0})$. Since ${\rm vol_{n-1}}(\overline K_p(f_{\lam_0})) = \om_{n-1}$, there is $B \in {\rm GL}_{n-1}$ such that ${\rm det}\, B =\pm 1$ and $\overline K_{p}(f_{\lam_0}) = B B_2^{n-1}$. Therefore, $\|x\|_{\overline K_p(f_{\lam_0})} = |B^{-1} x|$ which gives the following expression of $f_{\lam_0}$
\[
 f_{\lam_0}(t,x) = c\, \lt(\lt(1+ \frac t\mu\rt)^q + \Big{|} \frac{B^{-1}(x-x_0) + tB^{-1}\vec{b}_0}\mu\Big{|}^q \rt)^{-\frac{a-p}p}
\]
and
\[
 f(t,x) = c\, \lt(\lt(1+ \frac t{\lam_0\mu}\rt)^q + \Big{|} \frac{B^{-1}(x-x_0) + tB^{-1}\vec{a}_0}\mu\Big{|}^q \rt)^{-\frac{a-p}p}.
\]
Let us now denote $\vec{a} =(a_1,\ldots,a_{n-1})^t: = \frac1\mu B^{-1}\vec{a}_0 \in \R^{n-1}$, and 
\[
A = 
\begin{bmatrix}
\frac1{\lam_0\mu} & 0 & \ldots&0\\
a_1 \\
\vdots &  &\frac1\mu B^{-1}\\
a_{n-1} \\
\end{bmatrix}
\in {\rm SGL}_{n,+}.
\]
Obviously, we have $f(t,x) = c\, h_p(A(t, x-x_0))$ as wanted. The proof of Theorem \ref{affineGNtrace} is completed.\\

We finish this paper by remark that the arguments in this paper can be applied to obtain a more general sharp Gagliardo--Nirenberg trace inequality as follows. Let $p\in (1,n)$, $\lam \in (0,1)$ and $f\in \mathcal D^p(\R^n_+)$, define
\[
\mathcal E_{\lam,p}(f) =2^{\frac1p} c_{n-1,p} \lt(\int_{S^{n-2}} \lt(\int_{\R^n_+} \lt(\lam \la u, \na_x f\ra_+^p + (1-\lam)\la u,\na_x f\ra_-^p \rt)dx dt\rt)^{-\frac{n-1}{p}} du\rt)^{-\frac1{n-1}},
\]
where $a_+$ and $a_-$ denote the positive and negative parts of a number $a$, and
\[
\mathcal A_{\lam,p}(f)^p =c_{n,p,a}^p (\mathcal E_{\lam,p}(f))^{p-1 -\frac{(p-1)(a-n)}{a-1}} \inf_{\vec{a} \in \R^{n-1}} \lt(\int_{\R^n_+} |D_{(1,\vec{a})} f|^p dx dt\rt)^{\frac1p(1 +\frac{(p-1)(a-n)}{a-1})}.
\] 
From the definition, we have $\mathcal E_{\frac12,p}(f) = \mathcal E_p(f)$ and $\mathcal A_{\frac12,p} = \mathcal A_p(f)$.  Applying the general $L_p$ Busemann--Petty centroid inequality \cite{HS09Iso} and the argument in this section, we can prove the following inequality whose detail proof is left for interest reader,
\begin{equation*}
\lt(\int_{\pa \R^n_+} |f(0,x)|^{\frac{p(a-1)}{a-p}} dx\rt)^{\frac{a-p}{p(a-1)}} \leq \mathcal D_{n,a,p}\,  \mathcal A_{\lam,p}(f)^{\theta} \lt(\int_{\R^n_+} |f(t,x)|^{p\frac{a-1}{a-p}} dxdt\rt)^{(1-\theta)\frac{a-p}{p(a-1)}}
\end{equation*}
for any $f \in \mathcal D^{p, \frac{p(a-1)}{a-p}}(\R^n_+)$, where $\theta$ is given by \eqref{eq:theta} and $\mathcal D_{n,a,p}$ is given by \eqref{eq:sharpconstantaffineGN}. Moreover, there is equality if and only if 
\begin{equation*}
f(t,x) = c\, h_p(A(t, x-x_0))
\end{equation*}
for some $c \in \R$, $x_0\in \R^{n-1}$ and $A \in {\rm SGL}_{n,+}$, where $h_p(t,x) =((1+t)^q + |x|^q)^{-\frac{a-p}{p-1}}$.


\section*{Acknowledgments}
This work was supported by the CIMI postdoctoral research fellowship.

\end{document}